\documentclass{amsart}
\usepackage{latexsym, amssymb, amsmath}

\newtheorem{conj}{Conjecture}
\newtheorem{thm}{Theorem}[section]
\newtheorem{lemm}[thm]{Lemma}
\newtheorem{cor}[thm]{Corollary}
\newtheorem{prop}[thm]{Proposition}
\theoremstyle{remark}
\newtheorem{rmk}[thm]{Remark}
\theoremstyle{definition}
\newtheorem{defi}[thm]{Definition}

\newtheorem*{ex}{Example}

\newcommand{\lapla}{\bigtriangleup}

\newcommand{\p}{\phi}

\title{Biminimal properly immersed submanifolds \\ in the Euclidean spaces}

\author{Shun Maeta}

\thanks{
Division of Mathematics, GSIS, Tohoku University, 
Sendai 980-8579, Japan.~\\
supported in part by 
Research Fellowships of the Japan Society for the Promotion of Science for Young Scientists, No.~23-6949.~\\
e-mail:~shun.maeta@gmail.com}
\thanks{2010~{\em Mathematics Subject Classification.}~primary 58E20, secondary 53C43, 53A07}
\date{January, 13,  2012.} 

\begin{document} 
\maketitle 
\markboth{Biminimal properly immersed submanifolds in the Euclidean spaces} 
{Shun Maeta}

\begin{abstract} 
We consider a {\it complete nonnegative biminimal} submanifold  $M$ (that is, a complete biminimal submanifold with $\lambda\geq0$)
in a Euclidean space $\mathbb{E}^N$. 
Assume that the immersion is {\it proper}, that is, 
the preimage of every compact set in $\mathbb{E}^N$ is also compact in $M$. 
Then, we prove that $M$ is minimal. 
From this result, we give an affirmative partial answer to Chen's conjecture.
For the case of $\lambda<0$, we construct examples of biminimal submanifolds and curves.
\end{abstract}

\section{\bf Introduction}\label{intro} 
Theory of harmonic maps has been applied into various fields in differential geometry.
 The harmonic maps between two Riemannian manifolds are
 critical maps of the energy functional $E(\p)=\frac{1}{2}\int_M\|d\p\|^2v_g$, for smooth maps $\p:(M^n,g)\rightarrow (\tilde{M}^N,\langle\ ,\ \rangle)$.

On the other hand, in 1981, J. Eells and L. Lemaire \cite{jell1} proposed the problem to consider the {\em polyharmonic maps of order k{ $($\em{$k$-harmonic maps}$)$:
 they are critical maps of the functional 
 \begin{align*}
 E_{k}(\p)=\int_Me_k(\p)v_g,\ \ (k=1,2,\dotsm),
 \end{align*}
 where $e_k(\p)=\frac{1}{2}\|(d+d^*)^k\p\|^2$ for smooth maps $\p:(M^n,g)\rightarrow (\tilde{M}^N,\langle\ ,\ \rangle)$.
G.Y. Jiang \cite{jg1} studied the first and second variational formulas of the bi-energy $E_2$, 
and critical maps of $E_2$ are called {\em biharmonic maps} ({\em 2-harmonic maps}). There have been extensive studies on biharmonic maps.
The Euler-Lagrange equation of $E_2$ is 
$$\tau_2(\p):=-\lapla^\p\tau(\p)-\sum^n_{i=1}R^{\tilde{M}}(\tau(\p),d\p(e_i))d\p(e_i)=0,$$
where 
$\lapla^{\p}:=\sum^n_{i=1}(\nabla^{\p}_{e_i}\nabla^{\p}_{e_i}-\nabla^{\p}_{\nabla_{e_i}e_i})$, 
$\tau(\p):=\text{trace}\nabla d\p$, $R^{\tilde{M}}$ and $\{e_i\}$ are the rough Laplacian, the tension field of $\p$, the Riemannian curvature of $\tilde{M}$ i.e., $R^{\tilde{M}}(X,Y)Z:=[\nabla_{X},\nabla_Y]Z-\nabla_{[X,Y]}Z$ for any vector field X, Y and Z on $\tilde{M}$, and a local orthonormal frame field of $M$, respectively.
If an isometric immersion $\p:(M,g)\rightarrow (\tilde{M},\langle \ , \ \rangle)$ is biharmonic, then $M$ is called {\em biharmonic submanifold}.

For biharmonic submanifolds, 
there is an interesting problem, namely, Chen's Conjecture 
(cf.~\cite{Chen}): 

\begin{conj} 
Any biharmonic submanifold $M$ in $\mathbb{E}^N$ is minimal. 
\end{conj} 

There are many affirmative partial answers to Conjecture~$1$ 
(cf.~\cite{Chen, Chen-Ishikawa-1, Chen-Ishikawa-2, Leuven, Dimi, Hasanis-Vlachos}). 
In particular, there are some complete affirmative answers 
if $M$ is one of the following: 
(a) a curve \cite{Dimi}, 
(b) a surface in $\mathbb{E}^3$ \cite{Chen}, 
(c) a hypersurface in $\mathbb{E}^4$ \cite{Leuven, Hasanis-Vlachos}. 

On the other hand, 
since there is no assumption of {\it completeness} for submanifolds in Conjecture~1, 
in a sense it is a problem in {\it local} differential geometry.  
Recently, we reformulated Conjecture~1 into a problem 
in {\it global} differential geometry as the following (cf.~\cite{kasm1, N-U-1, N-U-2}):  

\begin{conj} 
Any {\rm complete} biharmonic submanifold in $\mathbb{E}^N$ is minimal. 
\end{conj} 

An immersed submanifold $M$ in $\mathbb{E}^N$ is said to be {\it properly immersed} 
if the immersion $M \rightarrow {\mathbb{E}}^N$ is a proper map. 
K. Akutagawa and the author showed that biharmonic properly immersed submanifold in the Euclidean space is minimal \cite{kasm1}.
Here, we remark that the properness of the immersion implies the completeness of $(M, g)$.

Recently, E. Loubeau and S. Montaldo introduced {\em biminimal immersion} :

\begin{defi}[\cite{elsm1}]
An immersion $\p:(M^n,g)\rightarrow (\tilde{M}^N,\langle \ , \ \rangle)$, $n\leq N$ is called {\em biminimal} if it is a critical point of the functional 
$$E_{2,\lambda}(\p)=E_2(\p)+\lambda E(\p),\ \ \lambda\in \mathbb{R}$$
 for any smooth variation of the map $\p_t (-\epsilon<t<\epsilon)$, $\p_0=\p$ such that $\left. V=\frac{d\p_t}{dt}\right |_{t=0}$ is normal to $\p(M)$.
\end{defi}

The Euler-Lagrange equation for biminimal immersion is
$$[\tau_2(\p)]^{\perp}+\lambda[\tau(\p)]^{\perp}=0,$$
where, $[\cdot]^{\perp}$ denotes the normal component of $[\cdot]$.
We call an immersion {\em free} biminimal if it is biminimal condition for $\lambda=0$.
(It is sometimes called that biminimal is $\lambda$-biminimal and free biminimal is biminimal, respectively).
 If $\p:(M,g)\rightarrow (\tilde{M},\langle \ ,\ \rangle)$ is an isometric immersion, then the biminimal condition is 
\begin{align}\label{biminimal eq}
[-\lapla^{\p}{\bf H}-\sum^n_{i=1}R^{\tilde{M}}({\bf H},d\p(e_i))d\p(e_i)]^{\perp}+\lambda{\bf H}=0,
\end{align}
for some $\lambda \in\mathbb{R}$.
If an isometric immersion $\p$ is biminimal, then $M$ is called {\em biminimal submanifold}.
\begin{rmk}
 we remark that every biharmonic submanifold is free biminimal one.
\end{rmk}

The remaining sections are organized as follows. 
Section~$2$ contains some necessary definitions and preliminary geometric results. 
In section~$3$, we show nonnegative biminimal properly immersed submanifold (that is, a biminimal properly immersed submanifold with $\lambda\geq0$) 
in the Euclidean space is minimal
 and get an affirmative partial answer to Chen's conjecture. 
 In section $\ref{lambda<0}$, we construct examples of biminimal submanifolds and curves for the case of $\lambda<0$. 

\vspace{10pt}

\noindent 
{\bf Acknowledgements.} 
The author would like to thank 
Kazuo Akutagawa for helpful discussions.

\section{\bf Preliminaries}\label{Pre} 

Let $M$ be an $n$-dimensional immersed submanifold in $\mathbb{E}^N$, 
${\bf x} : M \rightarrow \mathbb{E}^N$ its immersion  
and $g$ its induced Riemannian metric. 
For simplicity, we often identify $M$ with its immersed image ${\bf x}(M)$ in every local arguments. 
Let $\nabla$ and $D$ denote respectively the Levi-Civita connections 
of $(M, g)$ and $\mathbb{E}^N = (\mathbb{R}^N, \langle\ ,\ \rangle)$. 
For any vector fields $X, Y \in \frak{X}(M)$, 
the Gauss formula is given by 
$$ 
D_XY = \nabla_XY + h(X, Y), 
$$ 
where $h$ stands for the second fundamental form of $M$ in $\mathbb{E}^N$. 
For any normal vector field $\xi$, the Weingarten map $A_{\xi}$ with respect to $\xi$ 
is given by 
$$ 
D_X\xi = - A_{\xi}X + \nabla^{\perp}_X\xi, 
$$ 
where $\nabla^{\bot}$ stands for the normal connection of the normal bundle of $M$ in $\mathbb{E}^N$. 
It is well known that $h$ and $A$ are related by 
$$ 
\langle h(X, Y), \xi \rangle = \langle A_{\xi}X, Y \rangle. 
$$ 

For any $x \in M$, 
let $\{e_1, \cdots, e_n, e_{n+1}, \cdots, e_N\}$ be an orthonormal basis of $\mathbb{E}^N$ at $x$ 
such that $\{e_1, \cdots, e_n\}$ is an orthonormal basis of $T_xM$. 
Then, $h$ is decomposed as at $x$ 
$$ 
h(X, Y) = \Sigma_{\alpha=n+1}^N h_{\alpha}(X, Y)e_{\alpha}. 
$$ 
The mean curvature vector ${\bf H}$ of $M$ at $x$ is also given by 
$$ 
{\bf H}(x) = \frac{1}{n} \Sigma_{i = 1}^n h(e_i, e_i) = \Sigma_{\alpha=n+1}^N H_{\alpha}(x)e_{\alpha},\qquad 
H_{\alpha}(x) := \frac{1}{n} \Sigma_{i = 1}^n h_{\alpha}(e_i, e_i).  
$$ 
It is well known that the necessary and sufficient conditions for $M$ in $\mathbb{E}^N$ 
to be biharmonic, namely $\Delta {\bf H} = 0$, 
are the following (cf.~\cite{Chen, Chen-Ishikawa-1, Chen-Ishikawa-2}): 
\begin{equation}\label{N-S} 
\begin{cases} 
\ \ \Delta^{\perp} {\bf H} - \Sigma_{i=1}^n h(A_{\bf H}e_i, e_i) = 0, \\ 
\ \ n~\nabla |{\bf H}|^2 + 4~{\rm trace}~A_{\nabla^{\perp} {\bf H}} = 0, \\  
\end{cases} 
\end{equation} 
where $\Delta^{\perp}$ is the (non-positive) Laplace operator associated with the normal connection $\nabla^{\perp}$. 
Similarly, the necessary and sufficient condition for $M$ in $\mathbb{E}^N$ to be biminimal is the following:
\begin{equation}\label{biminimal eu} 
\ \ \Delta^{\perp} {\bf H} - \Sigma_{i=1}^n h(A_{\bf H}e_i, e_i) = \lambda{\bf H}. \\ 
\end{equation}

\vspace{10pt}


\section{Non existence theorem for biminimal submanifold}\label{EC} 
In this section, we show that a {\em nonnegative} biminimal properly immersed submanifold (that is, a biminimal properly immersed submanifold with $\lambda \geq0$) in the Euclidean space is minimal.

From the equation of (\ref{biminimal eu}), we have the following. 

\begin{lemm} 
Let $M = (M, g)$ be a nonnegative biminimal submanifold in $\mathbb{E}^N$. 
Then, the following inequality for $|{\bf H}|^2$ holds 
\begin{equation}\label{key} 
\Delta |{\bf H}|^2 \geq \frac{2}{n} |{\bf H}|^4. 
\end{equation} 
\end{lemm}

\begin{proof} 
The equation of (\ref{biminimal eu}) implies that, at each $x \in M$, 
\begin{align}\label{ell-ineq} 
\Delta |{\bf H}|^2 
& = 2~\Sigma_{i=1}^n \langle \nabla_{e_i}^{\perp} {\bf H}, \nabla_{e_i}^{\perp} {\bf H} \rangle 
+ 2~\langle \Delta^{\perp} {\bf H}, {\bf H} \rangle \notag \\ 
& = 2~\Sigma_{i=1}^n \langle \nabla_{e_i}^{\perp} {\bf H}, \nabla_{e_i}^{\perp} {\bf H} \rangle 
+2~\Sigma_{i=1}^n \langle h(A_{\bf H} e_i, e_i), {\bf H} \rangle
+ 2~\lambda \langle {\bf H}, {\bf H} \rangle  \\  
& \geq 2~\Sigma_{i=1}^n \langle A_{\bf H} e_i, A_{\bf H} e_i \rangle .\notag   
\end{align} 
When ${\bf H}(x) \ne 0$, set $e_N := \frac{{\bf H}(x)}{|{\bf H}(x)|}$. 
Then, ${\bf H}(x) = H_N(x) e_N$ and $|{\bf H}(x)|^2 = H_N(x)^2$. 
From (\ref{ell-ineq}), we have at $x$ 
\begin{align}\label{ell-ineq} 
\Delta |{\bf H}|^2 
& \geq 2~H_N^2~\Sigma_{i=1}^n \langle A_{e_N} e_i, A_{e_N} e_i \rangle\notag \\ 
& = 2~|{\bf H}|^2~|h_N|_g^2 \\
& \geq \frac{2}{n}~|{\bf H}|^4 \notag. 
\end{align} 
Even when ${\bf H}(x) = 0$, the above inequality~(\ref{key}) still holds at $x$. 
This completes the proof. 
\end{proof}

 \begin{thm}
Any nonnegative biminimal properly immersed submanifold in $\mathbb{E}^N$ is minimal.
 \end{thm}

\begin{proof}
If $M$ is compact, applying the standard maximum principle to the elliptic inequality (\ref{key}), 
we have that ${\bf H} = 0$ on $M$. 
Therefore, we may assume that $M$ is noncompact. 
Suppose that ${\bf H}(x_0) \ne 0$ at some point $x_0 \in M$. 
Then, we will lead a contradiction. 

Set 
$$ 
u(x) := |{\bf H}(x)|^2\quad {\rm for}\ \ x \in M. 
$$  
For each $\rho > 0$, consider the function 
$$ 
F(x) = F_{\rho}(x) := (\rho^2 - |{\bf x}(x)|^2)^2 u(x)\quad {\rm for}\ \ 
x \in M \cap {\bf x}^{-1}\big{(} \overline{{\bf B}_{\rho}} \big{)}.  
$$ 
Then, there exists $\rho_0 > 0$ such that $x_0 \in {\bf x}^{-1}\big{(} {\bf B}_{\rho_0} \big{)}$. 
For each $\rho \geq \rho_0$, 
$F = F_{\rho}$ is a nonnegative function which is not identically zero 
on $M \cap {\bf x}^{-1}\big{(} \overline{{\bf B}_{\rho}} \big{)}$. 
Take any $\rho \geq \rho_0$ and fix it. 
Since $M$ is properly immersed in $\mathbb{E}^N$, $M \cap {\bf x}^{-1}\big{(} \overline{{\bf B}_{\rho}} \big{)}$ is compact. 
By this fact combined with $F = 0$ on $M \cap {\bf x}^{-1} \big{(} \partial\overline{{\bf B}_{\rho}} \big{)}$, 
there exists a maximum point $p \in M \cap {\bf x}^{-1}\big{(} {\bf B}_{\rho} \big{)}$ of $F = F_{\rho}$ such that $F(p) > 0$. 
We have $\nabla F = 0$ at $p$, and hence 
\begin{equation}\label{grad}
\frac{\nabla u}{u} = \frac{2~\nabla |{\bf x}(x)|^2}{\rho^2 - |{\bf x}(x)|^2}\quad {\rm at}\ \ p. 
\end{equation} 
We also have that $\Delta F \leq 0$ at $p$. 
Combining this with (\ref{grad}), we obtain 
\begin{equation}\label{Lap} 
\frac{\Delta u}{u} \leq \frac{6~|\nabla |{\bf x}(x)|^2|_g^2}{(\rho^2 - |{\bf x}(x)|^2)^2} 
+ \frac{2~\Delta |{\bf x}(x)|^2}{\rho^2 - |{\bf x}(x)|^2}\quad {\rm at}\ \ p.   
\end{equation}   
From $\Delta {\bf x}=n{\bf H}$, we note 
\begin{equation}\label{sub} 
\begin{cases} 
\ \ \Delta |{\bf x}(x)|^2 = 2~\Sigma_{i=1}^n |\nabla_{e_i} {\bf x}(x)|^2 + 2~\langle \Delta {\bf x}(x), {\bf x}(x) \rangle 
\leq 2 n + 2 n |{\bf H}|\cdot |{\bf x}(x)|, \\ 
\ \ |\nabla |{\bf x}(x)|^2|_g^2 \leq 4 n |{\bf x}(x)|^2.   
\end{cases} 
\end{equation} 
It then follows from (\ref{key}), (\ref{Lap}) and (\ref{sub}) that 
$$ 
u(p) \leq \frac{12 n^2 |{\bf x}(p)|^2}{(\rho^2 - |{\bf x}(p)|^2)^2} 
+ \frac{2 n^2 (1 + \sqrt{u(p)} |{\bf x}(p)|)}{\rho^2 - |{\bf x}(p)|^2},  
$$ 
and hence 
$$ 
F(p) \leq 12 n^2 |{\bf x}(p)|^2 + 2 n^2 (\rho^2 - |{\bf x}(p)|^2) + 2 n^2 \sqrt{F(p)} |{\bf x}(p)|.  
$$ 
Therefore, there exists a positive constant $c(n) > 0$ depending only on $n$ such that 
$$ 
F(p) \leq c(n) \rho^2. 
$$ 

Since $F(p)$ is the maximum of $F = F_{\rho}$, we have 
$$ 
F(x) \leq F(p) \leq c(n) \rho^2\quad {\rm for}\ \ x \in M \cap {\bf x}^{-1}\big{(} \overline{{\bf B}_{\rho}} \big{)}, 
$$ 
and hence 
\begin{equation}\label{Final} 
|{\bf H}(x)|^2 = u(x) \leq \frac{c(n) \rho^2}{(\rho^2 - |{\bf x}(x)|^2)^2}\quad 
{\rm for}\ \ x \in M \cap {\bf x}^{-1}\big{(} {\bf B}_{\rho} \big{)}\quad {\rm and}\ \ \rho \geq \rho_0. 
\end{equation} 
Letting $\rho \nearrow \infty$ in (\ref{Final}) for $x = x_0$, 
we have that 
$$ 
|{\bf H}(x_0)|^2 = 0. 
$$  
This contradicts our assumption that ${\bf H}(x_0) \ne 0$.  
Therefore, $M$ is minimal. 
\end{proof}  

Especially, any free biminimal properly immersed submanifold in $\mathbb{E}^N$ is minimal.
From the equations $(\ref{N-S})$, we have:
\begin{cor}[\cite{kasm1}]
Any biharmonic properly immersed submanifold in $\mathbb{E}^N$ is minimal.
\end{cor}
This corollary gives an affirmative partial answer to Chen's conjecture.

\vspace{10pt}

\section{Biminimal submanifold with $\lambda<0$}\label{lambda<0}
For the case of $\lambda <0$, we shall construct biminimal submanifolds.

\begin{prop}[\cite{elsm1}]
Let $\p:M^n\rightarrow \mathbb{E}^{n+1}$ be an isometric immersion with ${\bf H}=H_{n+1}{e_{n+1}}$ its mean curvature vector. Then $M$
 is biminimal if and only if 
 $$\lapla H_{n+1}=(|h|^2+\lambda)H_{n+1},$$
for some value of $\lambda$ in $\mathbb{R}$.
\end{prop}

From this proposition, if $M$ is a non-trivial biminimal submanifold with harmonic mean curvature, then $\lambda<0.$

\begin{cor}
Let $\p:M^n\rightarrow \mathbb{E}^{n+1}$ be an isometric immersion with harmonic mean curvature. If  M is free biminimal, then it is minimal.
\end{cor}

\begin{cor}
Let $\p:M^n\rightarrow \mathbb{E}^{n+1}$ be an isometric immersion with harmonic mean curvature. Then $M$
 is non-trivial biminimal if and only if 
 $$|h|^2=-\lambda,$$
for $\lambda(<0) \in \mathbb{R}$.
\end{cor}

Using this result, we obtain following:

\begin{prop}
The isometric immersion $\p:S^{n}\left(\sqrt{\frac{n}{-\lambda}}\right)\rightarrow \mathbb{E}^{n+1},$ $(\lambda<0)$
 is non-trivial biminimal.  
\end{prop}

\begin{proof}
In this case, $A=-\frac{1}{\sqrt{\frac{n}{-\lambda}}}I$, where $I$ is the identity transformation. Therefore, we have  
$|h|^2=n\frac{1}{{\sqrt{\frac{n}{-\lambda}}}^2}=-\lambda.$
\end{proof}

For the curve case, we shall construct biminimal curves.

\begin{defi}[\cite{elsm1}]
The Frenet frame $\{B_i\}_{i=1,\cdots ,N}$ associated with a curve $\gamma :I\subset \mathbb{R}\rightarrow (\tilde{M},\langle \ , \ \rangle)$ is the orthonormalization of the 
$(N+1)$-tuple
$\left\{\nabla^{\gamma (k)}_{\frac{\partial}{\partial t}} d\gamma \left(\frac{\partial}{\partial t}\right)\right\}_{k=0,\cdots ,N}$
described by
\begin{align*}
B_1=&d\gamma \left( \frac{\partial}{\partial t}\right),\\
\nabla^{\gamma}_{\frac{\partial}{\partial t}} B_1=&k_1B_2,\\
\nabla^{\gamma}_{\frac{\partial}{\partial t}} B_i=&-k_{i-1}B_{i-1}+k_iB_{i+1}\ \ \ ^\forall i=2,\cdots ,N-1,\\
\nabla^{\gamma}_{\frac{\partial}{\partial t}} B_N=&-k_{N-1}B_{N-1},
\end{align*}
where the functions $\{k_1>0, k_2, k_3,\cdots , k_{N-1}\}$ are called the curvatures of $\gamma.$
Note that $B_1=\gamma'$ is the unit tangent vector field to the curve.
\end{defi}

Biminimal curves in a Euclidean space are characterized as follows.
\begin{prop}[\cite{elsm1}]
Let $\gamma:I\subset \mathbb{R} \rightarrow \mathbb{E}^N$, $N\geq2,$ be a curve parametrized by arc length from an open interval of $\mathbb{R}$ 
 into a Euclidean space $\mathbb{E}^N$.  Then $\gamma $ is biminimal if and only if $k_i$ fulfill the system:
 \begin{equation}
 \begin{cases}
 k''_1-k^3_1-k_1k_2^2-\lambda k_1=0,\\
 k_1^2k_2=\text{constant},\\
 k_1k_2k_3=0.
 \end{cases}
 \end{equation}
\end{prop}

When $\lambda <0$, using this proposition, we construct a example of biminimal curves.
\begin{ex}
We consider the curve 
$$\gamma (s)=\frac{1}{\sqrt{-\lambda}}\left\{\cos(\sqrt{-\lambda}s) c_1+\sin(\sqrt{-\lambda}s)c_2\right\}+c_3,\ \ (\lambda<0),$$
where, $c_1,c_2$ are constant vectors orthogonal to each other with $|c_1|^2=|c_2|^2=1$, and $c_3$ is a constant vector.
 Direct computation shows that the curve is non-trivial biminimal.
\end{ex}

\quad \\ 
\quad \\

\bibliographystyle{amsbook}

\vspace{20mm} 

\end{document}